\documentclass[11pt]{article}      
\usepackage{amsmath,amssymb,amscd,amsthm, graphicx, verbatim, setspace} 
\usepackage{authblk}

\theoremstyle{plain}
\newtheorem{thm}{Theorem}[section]
\newtheorem{lem}[thm]{Lemma}
\newtheorem{cor}[thm]{Corollary}
\newtheorem{rem}[thm]{Remark}
\newtheorem{rec}[thm]{Recurrence}

\title{Forbidden formations in 0-1 matrices}
\date{}
\author{Jesse Geneson\\
\small PSU\\
\small\tt geneson@gmail.com
}
\begin{document}
\maketitle
\begin{abstract}
Keszegh (2009) proved that the extremal function $ex(n, P)$ of any forbidden light $2$-dimensional 0-1 matrix $P$ is at most quasilinear in $n$, using a reduction to generalized Davenport-Schinzel sequences. We extend this result to multidimensional matrices by proving that any light $d$-dimensional 0-1 matrix $P$ has extremal function $ex(n, P,d) = O(n^{d-1}2^{\alpha(n)^{t}})$ for some constant $t$ that depends on $P$. 

To prove this result, we introduce a new family of patterns called $(P, s)$-formations, which are a generalization of $(r, s)$-formations, and we prove upper bounds on their extremal functions. In many cases, including permutation matrices $P$ with at least two ones, we are able to show that our $(P, s)$-formation upper bounds are tight.
\end{abstract}

\section{Introduction}

We say that 0-1 matrix $M$ \emph{contains} 0-1 matrix $P$ if some submatrix of $M$ either equals $P$ or can be turned into $P$ by changing some ones to zeroes. Otherwise we say that $M$ \emph{avoids} $P$. The function $ex(n, P)$ is defined as the maximum number of ones in any 0-1 matrix with $n$ rows and $n$ columns that avoids $P$. This function has been applied to many problems, including the Stanley-Wilf conjecture \cite{MT}, the maximum number of unit distances in a convex $n$-gon \cite{Fure}, bounds on the lengths of sequences avoiding forbidden patterns \cite{petmnl}, and robot navigation problems \cite{Mit}.

The 0-1 matrix extremal function has a lower bound of $n$ for all 0-1 matrices except those with all zeroes or just one entry, just by using an $n \times n$ matrix with either a single row or a single column of ones. Given this trivial lower bound, F\"{u}redi and Hajnal posed the problem of finding all 0-1 matrices $P$ such that $ex(n, P) = O(n)$ \cite{FH}.

There is only partial progress on this problem. Marcus and Tardos proved that $ex(n, P) = O(n)$ for every permutation matrix $P$ \cite{MT}, solving the Stanley-Wilf conjecture. Their proof used a divide-and-conquer approach to derive an inequality for $ex(n, P)$ in terms of $ex(n/k^2, P)$, where $k$ is the sidelength of $P$. There are other classes of 0-1 matrices that are known to have linear extremal functions, such as double permutation matrices \cite{G}, single rows or columns of ones, and 0-1 matrices with ones in a V-shape \cite{Ke}. 

Some 0-1 matrices have extremal functions that are very close to linear, but not exactly linear. We call a function $f(n)$ quasilinear if $f(n) = O(n 2^{\alpha(n)^{t}})$ for some constant $t$, where $\alpha(n)$ is the very slow-growing inverse Ackermann function. For example, both patterns below are known to have extremal functions on the order of $\theta(n \alpha(n))$, so their extremal functions are both quasilinear and nonlinear. 

$$
\begin{pmatrix}
\bullet &         & \bullet \\
        & \bullet &         & \bullet\\
\end{pmatrix}
\begin{pmatrix}
\bullet &         &         &         \\
        &         & \bullet &         \\
        & \bullet &         & \bullet\\
\end{pmatrix}
$$

In general, many 0-1 matrices are known to have quasilinear extremal functions. A 0-1 matrix is called light if it has at most a single one in each row. Keszegh proved that every light 0-1 matrix has a quasilinear extremal function \cite{Ke} using a simple transformation from 0-1 matrices to sequences. Note that we could have defined lightness with columns instead of rows, since $ex(n, P)$ is preserved under matrix transposition.

In the same paper, Keszegh proved the existence of infinitely many minimally non-quasilinear patterns, and this proof was later modified in \cite{G} to prove the existence of infinitely many minimally non-linear patterns.

Extremal functions of forbidden patterns have also been studied for higher-dimensional matrices. Define $ex(n, P, d)$ to be the maximum number of ones in a $P$-free $d$-dimensional 0-1 matrix of sidelength $n$. As with the $d = 2$ case, in general the $d$-dimensional extremal function has a lower bound of $n^{d-1}$ except for forbidden patterns with all zeroes or just one entry.

The natural generalization of F\"{u}redi and Hajnal's problem to $d$ dimensions is to find all $d$-dimensional $0-1$ matrices $P$ for which $ex(n, P, d) = O(n^{d-1})$. This property has been proved for all $d$-dimensional permutation matrices and double permutation matrices $P$ \cite{km, gt}. 

The natural extension of quasilinearity to $d$ dimensions replaces the $n$ with an $n^{d-1}$. Call a $d$-dimensional 0-1 matrix $P$ light if $P$ has no pair of ones with the same $i^{th}$ coordinate for all $i \neq d$. Note that we could have defined lightness differently using any of the other dimensions besides $i = d$ since $ex(n, P, d)$ is preserved under transposition of dimensions.

We generalize Keszegh's result to any number of dimensions by proving that any light $d$-dimensional matrix $P$ has extremal function on the order of $O(n^{d-1} 2^{\alpha(n)^{t}})$ for some constant $t$ that depends on $P$. To prove this result about light matrices, we define a new forbidden pattern called a $(P, s)$-formation, and a new extremal function for multidimensional 0-1 matrices. This new extremal function is similar to one used in a few earlier papers, but we modify the definition to obtain more useful properties that we do not have with the old definition.

Using an inequality relating the new extremal function and the standard extremal function $ex(n, P, d)$, we are able to derive upper bounds on the standard extremal functions $ex(n, F_{P, s}, d+1)$ of $(P, s)$-formations. In cases when $ex(n, P, d) = O(n^{d-1})$ and $P$ has at least two ones, such as when $P$ is a $j$-tuple permutation matrix, we show that our upper bounds are tight.

Our results also address a question from \cite{niv}: What other problems, like interval chains and almost-DS sequences, satisfy recurrences like Recurrences 4.7 and 4.11 from \cite{niv}? The new extremal functions that we define for $(P, s)$-formations provide examples of such problems.

\subsection{Multidimensional formations}

An $i$-row of a $d$-dimensional 0-1 matrix is a set of entries all of whose coordinates other than the $i^{th}$ coordinate have the same value. Given a $d$-dimensional 0-1 matrix $P$ with $r$ ones, a $(P, s)$-formation is a set of $s r$ ones in a $(d+1)$-dimensional 0-1 matrix that can be partitioned into $s$ disjoint groups $G_{1}, \dots, G_{s}$ of $r$ ones so that any two groups $G_{i}, G_{j}$ have ones in the exact same sets of $1$-rows, the maximum first coordinate of any one in $G_{i}$ is less than the minimum first coordinate of any one in $G_{j}$ if $i < j$, and $P$ is the pattern obtained from  treating the ones as a subset of the set of all entries, removing the first coordinate of each one, and only counting repeated elements once.

For each $d$-dimensional $0-1$ matrix $P$, define  $F_{P, s}$ to be the set of all $(P,s)$-formations. Observe that if $P$ is a $d$-dimensional $0-1$ matrix and $Q$ is a light $(d+1)$-dimensional $0-1$ matrix with $s$ ones that reduces to $P$ when the first coordinate is removed, then every element of the class $F_{P,s}$ contains $Q$, so $ex(n,Q,d+1) \leq  ex(n, F_{P, s}, d+1)$. We will also use the function $ex(n,m, F, d)$, which is the maximum number of ones in an $F$-free $d$-dimensional $0-1$ matrix with first dimension of length $m$ and other dimensions of length $n$.

To bound $ex(n, F_{P, s}, d+1)$, we use a technique analogous to the one used to bound extremal functions of forbidden $(r,s)$-formations in sequences \cite{niv}, fat formations in $2$-dimensional 0-1 matrices \cite{ck}, and tuples stabbing interval chains \cite{AKNSS}. In the case of sequences in \cite{niv}, the technique involves defining a new extremal function which maximizes the number of distinct letters in a formation-avoiding sequence with $m$ blocks, and using upper bounds on the new extremal function to derive upper bounds on the standard sequence extremal function. It should also be noted that Pettie derived sharp bounds on extremal functions of $(r, s)$-formations and doubled $(r, s)$-formations \cite{Pet1, Pet2} for cases that were not already covered in \cite{niv, ck} using a technique very different from the one in \cite{niv} and this paper.

\subsection{New extremal function}

The new sequence extremal function from \cite{niv} was extended to 0-1 matrices in \cite{ck} and \cite{gs}. Those papers defined a matrix extremal function $ex_{k}(m, P)$ which is the maximum number of columns in a $P$-free 0-1 matrix with $m$ rows in which every column has at least $k$ ones. Note that $P$ can be a single pattern or a family of patterns.

Compared to the new sequence extremal function from \cite{niv}, the definition of $ex_{k}(m, P)$ replaces letters with columns and occurrences of letters with ones. However, there is a more natural definition to use for $ex_{k}(m, P)$ in order to prove the results in this paper. With this new definition, some of the results in \cite{gs} and \cite{ck} can be strengthened, or at least obtained as corollaries. 

In this paper, we define $lx_{k}(m, P)$ to be the maximum possible number of distinct letters in a matrix with $m$ rows in which each letter occurs at least $k$ times, each letter has all of its occurrences in a single column, and the $0-1$ matrix obtained from changing all letters to ones and all blanks to zeroes avoids $P$.

Before defining the multidimensional version of $lx_{k}(m, P)$, we prove a few basic facts about $lx_{k}(m, P)$. The first fact below is similar to Theorem 1 in \cite{gs}, which is restricted to a class of patterns called range-overlapping, but the lemma below has a more general statement than Theorem 1 in \cite{gs} since the new definition of $lx_{k}(m, P)$ more closely parallels the definition of the corresponding extremal function for sequences from \cite{niv}. 

\begin{lem}\label{relate0}
For any 0-1 matrix or family of 0-1 matrices $P$, $ex(n, m, P) \leq k(lx_{k}(m, P)+n)$. 
\end{lem}

\begin{proof}
Let $A$ be a 0-1 matrix with $m$ rows, $n$ columns, and $ex(n, m, P)$ ones which avoids $P$. For any column of $A$, divide the ones in the column into groups of size $k$, write a new letter on each group of ones, and delete at most $k-1$ ones in each column that are left over if the number of ones in the column is not divisible by $k$. 

The new matrix has each letter occurring $k$ times with its occurrences in the same column, and the $0-1$ matrix obtained from changing all letters to ones and all blanks to zeroes is contained in $A$, so it avoids $P$. Since we deleted at most $(k-1)n$ ones, we have $ex(n, m, P) \leq k(lx_{k}(m, P)+n)$.
\end{proof}

The next lemma strengthens Lemma 2 from \cite{gs}, since $lx_{k}(m, P)$ is never less than $ex_{k}(m, P)$.

\begin{lem}
If $P$ is a family of 0-1 matrices, $c$ is a constant, and $g$ satisfies $ex(n, m, P) \leq g(m)+c n$ for all $m, n$, then $lx_{k}(m, P) \leq \frac{g(m)}{k-c}$ for all $k > c$.
\end{lem}

\begin{proof}
Let $A$ be any $P$-avoiding matrix with $m$ rows and $lx_{k}(m, P)$ letters, in which each letter occurs $k$ times and makes all of its occurrences in the same column. Note that the number of nonempty columns in $A$ is at most $lx_{k}(m, P)$, so $A$ has at most $ex(lx_{k}(m, P), m, P) \leq g(m)+ lx_{k}(m, P) c$ ones by assumption. 

Thus $lx_{k}(m, P) k \leq g(m)+lx_{k}(m, P) c$, so we have $lx_{k}(m, P) \leq \frac{g(m)}{k-c}$ for all $m$.
\end{proof}

Observe that the last lemma can be applied to any pattern $P$ with a known linear extremal function $ex(n, P)$ that has ones in multiple columns, such as any permutation matrix or double permutation matrix, to obtain $\theta(m/k)$ bounds on $lx_{k}(m, P)$.

Define $R_{s}(j)$ as in \cite{niv} for $s \geq 1, j \geq 2$ by $R_{1}(j) = 2, R_{2}(j) = 3$, and for $s \geq 3$ define $R_{s}(2) = 2^{s-1}+1$ and $R_{s}(j) = R_{s}(j-1)R_{s-2}(j)+2R_{s-1}(j)-3R_{s-2}(j)-R_{s}(j-1)+2$ for $j \geq 3$. As in \cite{ck}, $D_{s}(j)$ is defined as the family of functions satisfying $D_{1}(j) = 0$, $D_{2}(j) = 2$, $D_{s}(2) = 2^{s-1}+2^{s-2}-1$ and $D_{s}(j) = 2D_{s-1}(j)+(D_{s-2}(j)+1)(R_{s}(j-1)-3)+D_{s}(j-1)-R_{s}(j-1)+1$.

Cibulka and Kyncl defined a doubled $(r, s)$-formation as a set $S$ of $r(2s-2)$ ones for which there exists an $s$-partition of the rows and an $r$-tuple of columns each of which has a single one from $S$ in the top and bottom rows of the partition and two ones in all other intervals \cite{ck}. Define $D_{r, s}$ as the set of all doubled $(r, s)$-formations. Cibulka and Kyncl proved for $m \geq k \geq D_{s}(j)$ that $ex_{k}(m, D_{r, s}) \leq c_{s} r m \alpha_{j}(m)^{s-2}$ where $c_{s}$ is a constant that depends only on $s$.

We note below the observation that the proof used in \cite{ck} gives the same upper bound for $lx_{k}$, with columns replaced by letters in the proof.

\begin{rem}
For $m \geq k \geq D_{s}(j)$ and every $j,r,s \geq 2$, $lx_{k}(m, D_{r, s}) \leq c_{s} r m \alpha_{j}(m)^{s-2}$ where $c_{s}$ is a constant that depends only on $s$.
\end{rem}

Note that Cibulka and Kyncl's result that $ex_{k}(m, D_{r, s}) \leq c_{s} r m \alpha_{j}(m)^{s-2}$ follows as a corollary from the last remark.

\section{Bounds on formations}

Given a forbidden family $F$ of $d$-dimensional 0-1 matrices, define $lx_{k}(n,m,F,d)$ to be the maximum possible number of distinct letters in a $d$-dimensional 0-1 matrix where the first dimension has length $m$ and the other dimensions have length $n$, every letter has at least $k$ occurrences and they are all in the same $1$-row, and the matrix obtained by replacing all letters with ones and all blanks with zeroes avoids all matrices in $F$. Given a $d$-dimensional 0-1 matrix $P$, define $G_{P,s}(n,m,d+1,k) = lx_{k}(n,m,F_{P,s},d+1)/ex(n,P,d)$.

Before deriving recurrences for $G_{P,s}(n,m,d+1,k)$, we first derive a general inequality between $ex(n,m, F,d)$ and $lx_{k}(n,m,F,d)$ which is very similar to Lemma \ref{relate0} and the inequality between the two sequence extremal functions in \cite{niv}. 

\begin{lem}\label{relate}
For all families $F$ of forbidden $d$-dimensional 0-1 matrices, $ex(n,m,F,d) \leq k(lx_{k}(n,m,F,d)+n^{d-1})$.
\end{lem}

\begin{proof}
This proof is nearly identical to the proof of Lemma \ref{relate0}, with $1$-rows replacing columns.
\end{proof}

The next proof below is very similar to one of the bounds in \cite{g1}, besides the factor of $ex(n, P,d)$. The $ex(n, P,d)$ can be seen as generalizing the factor of $r-1$ in that proof, since $ex(n, Q, 1) = r-1$ when $Q$ is the $1$-dimensional $0-1$ matrix of length $r$ with $r$ ones.

The $\binom{m-\lceil \frac{s}{2} \rceil}{\lfloor \frac{s}{2} \rfloor}$ upper bounds in \cite{g1} are stronger than the $\binom{m-2}{s-1}$ upper bounds for the same functions in \cite{niv}, but the rest of the proof of the main result after the next lemma will closely parallel \cite{niv}. In addition to being the first step of the main result, the next lemma also gives tight bounds on $ex(n,F_{P,3},d+1)$ for patterns $P$ which satisfy $ex(n, P, d) = \theta(n^{d-1})$.

\begin{lem}
$lx_{s}(n,m,F_{P,s},d+1) \leq ex(n, P, d) \binom{m-\lceil \frac{s}{2} \rceil}{\lfloor \frac{s}{2} \rfloor}$ for every $1 \leq s \leq m$.
\end{lem}

\begin{proof}
Let $A$ be a $(d+1)$-dimensional 0-1 matrix with first dimension of length $m$, other dimensions of length $n$, and $lx_{s}(n,m,F_{P,s},d+1)$ distinct letters such that $A$ avoids $(P,s)$-formations and every letter of $A$ has at least $s$ occurrences, all in the same $1$-row. An occurrence of a letter in a $1$-row of $A$ is called even if there are an odd number of occurrences of the letter to the left of it. Otherwise the letter is called odd. 

Suppose for contradiction that $A$ has at least $1+ex(n,P,d) \binom{m-\lceil \frac{s}{2} \rceil}{\lfloor \frac{s}{2} \rfloor}$ letters. The number of distinct tuples $(m_{1}, \ldots, m_{\lfloor \frac{s}{2} \rfloor})$ for which a letter could have even occurrences with $1$-coordinate $m_{1}, \ldots, m_{\lfloor \frac{s}{2} \rfloor}$ is equal to the number of positive integer solutions to the equation $(1+x_{1})+\ldots + (1+ x_{\lfloor \frac{s}{2} \rfloor})+x_{1+\lfloor \frac{s}{2} \rfloor}= m+1$ if $s$ is even and $(1+x_{1})+\ldots + (1+ x_{\lfloor \frac{s}{2} \rfloor})+x_{1+\lfloor \frac{s}{2} \rfloor}= m$ if $s$ is odd. 

For each $s$ the equation has $ \binom{m-\lceil \frac{s}{2} \rceil}{\lfloor \frac{s}{2} \rfloor}$ positive integer solutions, so there are at least $ex(n,P,d)+1$ distinct letters $q_{1}, \ldots, q_{ex(n,P,d)+1}$ with even occurrences having the same $\lfloor \frac{s}{2} \rfloor$ $1$-coordinates of $A$. Then $A$ contains a $(P,s)$-formation, a contradiction. 
\end{proof}

\begin{cor}
$G_{P,s}(n,m,d+1,s) \leq \binom{m-\lceil \frac{s}{2} \rceil}{\lfloor \frac{s}{2} \rfloor}$ for every $1 \leq s \leq m$.
\end{cor}

\begin{cor}
$ex(n,F_{P,3},d+1) \leq 3(ex(n,P,d) n+ n^{d})$
\end{cor}

The next several recurrences are proved in nearly the same way as the corresponding recurrences in \cite{niv}.

\begin{rec}
$G_{P,s}(n,2m,d+1,2k-1) \leq 2G_{P,s}(n,m,d+1,2k-1)+2G_{P,s}(n,2m,d+1,k)$ for every $s > 1$, $k$, and $m$.
\end{rec}

\begin{proof}
The proof that $lx_{2k-1}(n,2m,F_{P,s},d+1) \leq 2lx_{2k-1}(n,m,F_{P,s},d+1)+2lx_{k}(n,2m,F_{P,s},d+1)$ is the same as Recurrence 4.5 in \cite{niv}. Dividing both sides by $ex(n, P, d)$ gives the claim.
\end{proof}

\begin{cor}
For every fixed $s \geq 3$, $G_{P,s}(n,m,d+1,k) = O(m(\log{m})^{s-3})$ if $k = 2^{s-2}+1$.
\end{cor}

\begin{proof}
See Corollary 4.6 in \cite{niv}.
\end{proof}

Like the last proof and corollary, the proofs of the next recurrence and corollary parallel Recurrence 4.11 and Corollary 4.12 in \cite{niv}.

\begin{rec}
If $s \geq 3$ is fixed, $t$ is a free parameter, and $k, k_1, k_2, k_3$ are integers satisfying $k_{2}k_{3}+2k_{1}-3k_{2}-k_{3}+2 = k$, then 

\noindent $G_{P,s}(n,m,d+1,k) \leq 
(1+\frac{m}{t})(G_{P,s}(n,t,d+1,k)+2G_{P,s-1}(n,t,d+1,k_1)+G_{P,s-2}(n,t,d+1,k_2))+G_{P,s}(n,1+\frac{m}{t},d+1,k_3)$.
\end{rec}

\begin{cor}
For every $s \geq 2, j \geq 2$, and $k \geq R_{s-1}(j)$, $G_{P,s}(n,m,d+1,k) \leq c m \alpha_{j}(m)^{s-3}$, where $c$ only depends on $s$.
\end{cor}

\begin{cor}
For every $s \geq 2, j \geq 2$, and $k \geq R_{s-1}(j)$, $lx_{k}(n,m,F_{P,s},d+1) \leq c m \alpha_{j}(m)^{s-3} ex(n, P, d)$.
\end{cor}

Combining this last corollary with Lemma \ref{relate}, we obtain the following general bounds for any $P$ and $s$.

\begin{thm}
For all $k$, $ex(n,F_{P,s},d+1) \leq C_{s,k} (n \alpha_{k}(n)^{s-3} ex(n, P, d)+n^{d})$ for constants $C_{s,k}$ of the form

\[C_{s,k} = \begin{cases} 
      2^{(1/t!)k^{t} \pm O(k^{t-1})} & s \text{ odd} \\
      2^{(1/t!)k^{t}\log_{2}{k} \pm O(k^{t})} & s \text{ even}
   \end{cases}
\] where $t = [(s-3)/2]$.
\end{thm}

\begin{proof}
Note that \[R_{s}(j) = \begin{cases} 
      2^{(1/t!)j^{t} \pm O(j^{t-1})} & s \text{ even} \\
      2^{(1/t!)j^{t}\log_{2}{j} \pm O(j^{t})} & s \text{ odd}
   \end{cases}
\] where $t = [(s-2)/2]$. Thus this theorem follows from applying Lemma \ref{relate} to the bound in the last corollary with $k = R_{s-1}(j)$.
\end{proof}

Plugging in $k = \alpha(n)$ to the last theorem gives the main result.

\begin{cor}\label{mainc}
\[ex(n,F_{P,3},d+1) = O(ex(n, P, d)n), ex(n,F_{P,4},d+1) = O(n \alpha(n) ex(n, P, d)),\]
\[ ex(n,F_{P,5},d+1) = O(n 2^{\alpha(n)} ex(n, P, d)), \text{and for } s > 5 \text{ we have } \]
\[ex(n,F_{P,s},d+1) \leq \begin{cases} 
      ex(n, P, d) n 2^{(1/t!)\alpha(n)^{t} + O(\alpha(n)^{t-1})} & s \text{ odd} \\
      ex(n, P, d) n 2^{(1/t!)\alpha(n)^{t}\log_{2}{\alpha(n)} + O(\alpha(n)^{t})} & s \text{ even}
   \end{cases}
\] where $t = [(s-3)/2]$.
\end{cor}

Using the last corollary, we prove inductively that any light $d$-dimensional 0-1 matrix $P$ has extremal function at most $O(n^{d-1} 2^{\alpha(n)^{t}})$ for some constant $t$ that depends on $P$.

\begin{thm}
If $P$ is a light $d$-dimensional 0-1 matrix, then $ex(n, P, d) = O(n^{d-1} 2^{\alpha(n)^{t}})$ for some constant $t$ that depends on $P$.
\end{thm}
\begin{proof}
For $d = 2$, the result was proved in \cite{Ke}. Suppose that $d > 2$ and the result is true for all light $r$-dimensional 0-1 matrices with $r < d$. Let $P'$ be the light $(d-1)$-dimensional matrix obtained from $P$ by treating the ones as a subset of the set of all entries, removing the first coordinate of each one, and only counting repeated elements once. If $s$ is the number of ones in $P$, then $P$ is contained in all $(P',s)$-formations. 

By the inductive hypothesis, $ex(n, P',d-1) = O(n^{d-2} 2^{\alpha(n)^{t}})$ for some constant $t$ that depends on $P'$. By the preceding corollary, $ex(n,F_{P',s},d) = O(n^{d-1} 2^{\alpha(n)^{t'}})$ for some constant $t'$ that depends on $P'$ and $s$. Thus $ex(n, P, d) = O(n^{d-1} 2^{\alpha(n)^{t'}})$ for some constant $t'$ that depends on $P$.
\end{proof}

We also mention that the upper bounds on extremal functions of doubled $(r, s)$-formations from \cite{ck} can be generalized to $(d+1)$-dimensional doubled $(P, s)$-formations in a similar way to how the bounds from \cite{niv} were generalized to $(d+1)$-dimensional $(P,s)$-formations in this paper.

We finish with an observation that our upper bounds on $(P, s)$-formations are tight in many cases. In particular, consider any $d$-dimensional 0-1 matrix $P$ with at least two ones such that $ex(n, P, d) = O(n^{d-1})$. 

Let $A_{t}$ denote the $2 \times t$ matrix with $A_{t}(i, j) = 1$ if $i+j$ is even and $A_{t}(i, j) = 0$ if $i+j$ is odd. Note that $ex(n, A_{4}, 2) = \theta(n \alpha(n))$, $ex(n, A_{5}, 2) = \theta(n 2^{\alpha(n)})$, and $ex(n, A_{2t+3}, 2) = \Omega(n 2^{(1-o(1))\alpha^{t}(n)/t!})$ \cite{petdegnon}.

Every $(P, s)$-formation has a $2$-dimensional projection that contains $A_{s}$. Thus we have the tight results below by Corollary \ref{mainc}.

\begin{thm} If $P$ is a $d$-dimensional 0-1 matrix with at least two ones such that $ex(n, P, d) = O(n^{d-1})$, then
\[ ex(n,F_{P,4},d+1) = \theta(n \alpha(n) ex(n, P, d)), ex(n,F_{P,5},d+1) = \theta(n 2^{\alpha(n)} ex(n, P, d)),\]
\[ \text{and } ex(n,F_{P,2t+3},d+1) =  ex(n, P, d) n 2^{(1/t!)\alpha(n)^{t} \pm o(\alpha(n)^{t})}
\] 
\end{thm}

One remaining problem is to find sharp bounds on $ex(n, F_{P,s},d+1)$ for all $P$ and $s$, since the last theorem only covers $d$-dimensional patterns $P$ with at least two ones such that $ex(n, P, d) = O(n^{d-1})$. Another problem is to find better bounds on the extremal function $lx_{k}(m, P)$ and its higher-dimensional generalizations. Also, bounds on formations were used in \cite{fw1, fw2, fw3} to algorithmically find tight upper bounds on extremal functions of forbidden sequences and forbidden families of 0-1 matrices. Can this algorithmic method be extended to $(P, s)$-formations to obtain tight upper bounds on forbidden families of $d$-dimensional 0-1 matrices for $d > 2$?

\end{document}